\documentclass[12pt]{amsart}

\usepackage{graphicx,cite}
\usepackage{braket}
\usepackage{color}
\usepackage{geometry}
\usepackage{marginnote}
\usepackage{tikz-cd,mathdots}
\newtheorem{theorem}{Theorem}[section]
\newtheorem{lemma}[theorem]{Lemma}

\newtheorem{definition}[theorem]{Definition}

\newtheorem{remark}[theorem]{Remark}

\theoremstyle{remark}

\newcommand\G{\Gamma}
\newcommand\beq{\begin{equation}}
\newcommand\eeq{\end{equation}}
\newcommand\bed{\begin{definition}}
\newcommand\ed{\end{definition}}
\newcommand\brem{\begin{remark}}
\newcommand\erem{\end{remark}}

\newcommand{\short}[5]{0\mapsto #1 \mathop{\mapsto}\limits^{#2} #3 \mathop{\mapsto}\limits^{#4} #5 \mapsto 0}
\newcommand{\ext}[2]{Ext^1(#1,#2)}

\title{Three-representation problem in Banach spaces}
\author{P.~Kuchment}
\dedicatory{Dedicated to the memory of Professors Carlos Berenstein and Mikhail Shubin,\\ dear friends and wonderful mathematicians}
\thanks{This work was supported in part by the NSF DMS grants \# 1517938 and \# 2007408.}
\begin{document}
\date{\today}
\maketitle
\begin{abstract}We provide the proof of a previously announced result that resolves the following problem posed by A.~A.~Kirillov. Let $T$ be a presentation of a group $\mathcal{G}$ by bounded linear operators in a Banach space $G$ and $E\subset G$ be a closed invariant subspace. Then $T$ generates in the natural way presentations $T_1$ in $E$ and $T_2$ in $F:=G/E$. What additional information is required besides $T_1, T_2$ to recover the presentation $T$? In finite-dimensional (and even in infinite dimensional Hilbert) case the solution is well known: one needs to supply a group cohomology class $h\in H^1(\mathcal{G},Hom(F,E))$. The same holds in the Banach case, if the subspace $E$ is complemented in $G$. However, every Banach space that is not isomorphic to a Hilbert one has non-complemented subspaces, which aggravates the problem significantly and makes it non-trivial even in the case of a trivial group action, where it boils down to what is known as the three-space problem. This explains the title we have chosen. A solution of the problem stated above has been announced by the author in 1976, but the complete proof, for non-mathematical reasons,  has not been made available. This article contains the proof, as well as some related considerations of the functor $Ext^1$ in the category \textbf{Ban} of Banach spaces.
\end{abstract}

\section*{Introduction}
Let us begin stating the problem addressed in this text, which (without the name we use here) can be found in \cite[pp. 117--118]{KirillovBook}.

Consider a representation $T$ of a group $\mathcal{G}$ in a Hilbert space $H$, with a closed invariant subspace $H_1\subset H$. The restriction of $T$ to $H_1$ we denote by $T_1$. Then the quotient- (or factor-) space $H_2:=H/H_1$ inherits a representation $T_2$. We are interested in additional information required for unique reconstruction of $T$ from known $T_1$ and $T_2$. The answer is well known (see the reference above): one needs to supply a cohomology class
\begin{equation}\label{E:coh}
h\in H^1(\mathcal{G},Hom(H_2,H_1)).
\end{equation}
Establishing this, one uses the fact that the subspace $H_1\subset H$ is complemented in $H$.
This construction also works if the spaces are Banach and the subspace $H_1$ is complemented in $H$ \cite{Moore,Pinc,KirillovBook}. However, in any Banach space that is not isomorphic to a Hilbert one, non-complemented subspaces exist \cite{LindComplement}. Thus, classifying short exact sequences
\begin{equation}\label{E:short}
0\mapsto H_1\mapsto H \mapsto H_2 \mapsto 0,
\end{equation}
or in other words studying the functor $Ext^1$ in the category \textbf{Ban} of Banach spaces becomes non-trivial, even in absence of a group action (see \cite[p. 118]{KirillovBook}). It is called the \textbf{three-space problem} in Banach space theory and has attracted a lot of attention in the last several decades (see, e.g. \cite{3space_book} and references therein). This motivates the name we use:\\

\textbf{The three-representation problem:}\\
\emph{What additional information is required for unique (up to equivalence) reconstruction of a (continuous) Banach representation $T$ from its sub-representation $T_1$ and quotient-representation $T_2$?}\\

As a footnote in \cite[p 118]{KirillovBook}) says, a solution of this problem was announced in the author's 1976 paper \cite{Kuch76}, although the complete proof has never appeared.

The goal of this text is to provide this proof, as well as related discussions of the functor $Ext^1$ in the category of Banach spaces. Although a long time has passed since, it looks like such a publication is still warranted\footnote{The construction here was triggered by the work \cite{Lind}, where in our terms non-triviality of $Ext^1(H,H)$ for a Hilbert space $H$ was shown. It also preceded by several years the closely related and well known ``twisted sum'' construction of \cite{KaltonPeck}, although the latter can treat also locally-convex cases.}.

It is clear that the case of a trivial group action has to be dealt with first. Thus the first Section provides the definition of the functor $Ext^1$ in the category \textbf{Ban} of Banach spaces and its representation through factor systems following \cite{Kuch76}. Section \ref{S:main} contains the statement and proof of the main result. An alternative (and wider known) representation of $Ext^1$ is considered in Section \ref{S:moreext}. A discussion of dimensionality of $Ext^1$ is provided in Section \ref{S:small}. Various remarks are collected in Section \ref{S:rem}.

\section{$EXT^1$ in the category of Banach spaces}\label{S:ext}
Homological techniques in theory of topological and in particular Banach vector spaces and Banach algebras have been used for quite a while (see, e.g., older works \cite{Brown,Mityagin,Palam,Kasparov}) and are well developed and frequently used now (e.g., \cite{Helem,HelemFA,Michor,Cigl,Cabello,CastMSb}). In particular, functors $Ext$ have been considered in the linear topological spaces setting (e.g., \cite{Palam,3space_book} and references therein). The ``three-space problem'' \cite{CCCFG,CastMSb,3space_book,CabelloHouston,Cast00,Cast01,Cast03,Cast07,Cast15,Cast17,Cast19,Cast96,Cast97,CastGroups00,KaltonPeck} addresses the question of what properties are inherited by a Banach space $G$ from its closed subspace $E$ and the corresponding quotient space $F:=G/E$ . The relation to the above three-representation problem was indicated in \cite{KirillovBook,Kuch76}.

The functor $Ext^1$, to put it simply, contains all information needed for recovery of a Banach space from its closed subspace and the corresponding quotient-space.

Let $X$ and $Y$ be Banach spaces. We denote by $L(X,Y)$ the space of all bounded linear operators mapping $X$ into $Y$, equipped with the operator norm topology.

\begin{definition}\label{D:extension}
For Banach spaces $E$ and $F$, a short exact sequence
\begin{equation}
0\mapsto E \mathop{\mapsto}\limits^i G \mathop{\mapsto}\limits^\sigma F \mapsto 0,
\end{equation}
where $G$ is a Banach space, $i\in L(E,G)$, and $\sigma\in L(G,F)$ is called an \textbf{extension of $E$ by $F$}.
\end{definition}
In other words, $i$ is an isomorphism of $E$ onto a closed subspace of $G$, and $\sigma$ is the quotient operator that ``identifies'' $G/i(E)$ with $F$.

There is a natural equivalence relation on the set of such extensions.
\begin{definition}\label{D:equiv}
Two extensions
\begin{equation}
0\mapsto E \mathop{\mapsto}\limits^{i_1} G_1 \mathop{\mapsto}\limits^{\sigma_1} F \mapsto 0
\end{equation}
and
\begin{equation}
0\mapsto E \mathop{\mapsto}\limits^{i_2} G_2 \mathop{\mapsto}\limits^{\sigma_2} F \mapsto 0
\end{equation}
are called \textbf{congruent}, if there exists an isomorphism $h\in L(G_1,G_2)$, such that the diagram
\begin{equation}\label{E:diagr}
\begin{tikzcd}
0 \arrow{r} &E \arrow{d}{I_E} \arrow{r}{i_1} &G_1 \arrow{d}{h} \arrow{r}{\sigma_1} &F \arrow{d}{I_F} \arrow{r} &0\\
0 \arrow{r} &E \arrow{r}{i_2} &G_2 \arrow{r}{\sigma_2} &F \arrow{r} &0
\end{tikzcd}
\end{equation}
is commutative, where $I_E$ (resp. $I_F$) is the identity operator on $E$ (resp. $F$).
\end{definition}

\begin{definition}\label{D:ext1}
The set of equivalence classes of extensions of $E$ by $F$ is denoted $\mathbf{Ext^1(E,F)}$.
\end{definition}

In homological algebra, where $Ext$ appears as the derived functor for $Hom$, there are several standard techniques of computing it. The two common ones involve injective and projective resolutions (if the category has enough injective and projective objects), see e.g. \cite{Cartan,maclane}. The category \textbf{Ban} of Banach spaces with bounded linear operators as morphisms is known to have plenty of injective and projective objects (see, e.g. \cite{3space_book,LindClasBan} and references therein). In \cite[Theorem 2]{Kuch76} the projective resolution approach was applied\footnote{Possibly, for the first time.} to \textbf{Ban} to represent $Ext^1$. There is, however, a third homological algebra technique of the so called factor systems \cite[Section 2 of Chaper 3]{maclane}, which was successfully developed in the Banach space case and applied to the three-representation problem in \cite{Kuch76}. The main idea here is rather simple: ``trivial'' extensions (i.e. those where $G\approx E\oplus F$, or equivalently $E$ is complemented in $G$) correspond to the situations when the quotient map $G\mapsto F$ has a linear bounded right inverse operator $p:F\mapsto G$. However, the well known Bartle-Graves theorem \cite{Bartle} (see also \cite[Section 1.0]{ZKKP}) claims existence of such right inverse, which is lacking only the additivity property. Thus one can try to use such mappings to classify extensions, which happens to be equivalent to a Banach space version of the factor systems in homological algebra. Correspondingly, the group cohomology class describing the extension of the representations would have its values in a class of non-linear mappings, rather than $L(F,E)$, with the ``degree of non-linearity'' being controlled by the element of $\ext{F}{E}$ corresponding to the particular extension of Banach spaces.

\subsection{Factor system of an extension}\label{S:factor}

We describe now, following\footnote{With a misprint corrected.} \cite{Kuch76} a Banach space analog of an extension construction familiar from homological algebra (see, e.g. \cite[Section 2 of Chaper 3]{maclane}).

\bed\label{D:factors}
A \textbf{factor system} from a Banach space $F$ to a Banach space $E$ is a strongly continuous mapping
$$
\phi: F\times F \mapsto E,
$$
such that for any $x, x_i\in F$, scalar $\lambda$, and natural $n$
\begin{enumerate}
  \item $\phi(\lambda x, \lambda y)=\lambda \phi (x,y)$;
  \item $\phi(x,y)=\phi(y,x)$;
  \item $\phi(x,0)=0$;
  \item $\phi(x,y)+\phi(x+y,z)=\phi(y,z)+\phi(x,y+z)$;
  \item $\|\sum_{k=1}^{k=n-1} \phi (\sum_{i=1}^k x_i,x_{k+1})\|\leq C(\phi)\sum_{i=1}^n\|x_i\|$
\end{enumerate}
\ed
The factor systems clearly form a vector space $S(F,E)$, which is a Banach space with respect to the norm $\|\phi\|=\inf C(\phi)$, where infimum is taken over constants in (5). The operators from $L(E)$ and $L(F)$ act naturally on $S(F,E)$ from the left and right correspondingly.

One can wonder whether there are enough of such mappings for them to be useful. The considerations in the text below answer this question in positive.

\bed\label{D:R}
We denote by $R(F,E)$ the vector space of all continuous, homogeneous, bounded\footnote{Boundedness here means an estimate $\|h(x)\|\leq C\|x\|$ with some $C>0$ and all $x\in F$. } mappings $h$ from $F$ to $E$. (Bartle-Graves mappings are of this type.)

Equipped with the norm
\beq
\|h\|:= \sup_{x\neq 0}\frac{\|h(x)\|}{\|x\|},
\eeq
$R(F,E)$ becomes a Banach space.
\ed

It is clear that there is an isometric embedding $L(F,E)\subset R(F,E)$.

We need one more definition:
\bed\label{D:rho}
The operator $\rho: R(F,E)\mapsto S(F,E)$ acts as follows:
\beq
\rho h(x,y):=h(x+y)-h(x)-h(y).
\eeq
I.e., it checks the ``level of non-additivity'' of $h$.
\ed
Exactness of the following sequence is clear:
\beq
0\mapsto L(F,E) \subset R(F,E) \mathop{\mapsto}\limits^\rho S(F,E).
\eeq
The mapping $\rho$ is not surjective in general, so we will be interested in its co-kernel
\beq
S(F,E)/(\rho R(F,E)).
\eeq
It is not hard to check that this quotient space defines a functor with the properties identical to those expected from $Ext^1$. This is not an accident:
\begin{theorem}\cite[Theorem 1]{Kuch76}\label{T:Fact_ext}
There is an isomorphism of the bi-functors
\beq\label{E:factorsystem}
\ext{E}{F}\approx S(F,E)/(\rho R(F,E)).
\eeq
\end{theorem}
This theorem was stated without a proof in the brief announcement \cite{Kuch76}, so we prove it here. We will establish isomorphism of the two spaces for fixed $E$ and $F$. Its functoriality can be easily traced through the construction.
\begin{proof}
Let $\alpha\in \ext{F}{E}$ have a representative
\beq
\short{E}{i}{G}{\sigma}{F}.
\eeq
According to the Bartle-Graves theorem \cite{Bartle} (see also \cite[Section 1.0]{ZKKP}), there exists a mapping $p\in R(F,G)$ right inverse\footnote{Notice that while $p$ is continuous and homogeneous, it cannot be additive (and thus linear), unless the subspace $E$ is complemented in $G$, which is a trivial case of no interest here. Existence of a Bartle-Graves mapping $p$ thus implies existence of a non-additive homogeneous continuous bounded projector on $E$.} to $\sigma$. Then $\rho p\in S(F,G)$. In fact, it belongs to $S(F,E)\subset S(F,G)$, since
\beq
\sigma\rho p(x,y)=\sigma(p(x+y)-p(x)-p(y))=x+y-x-y=0.
\eeq
We thus can map $\alpha$ to the equivalence class of $\rho p$ in $S(F,E)/(\rho R(F,E))$. One needs to check correctness of this definition, i.e. independence of the choice of a representative of $\alpha$ and of the choice of $d$ (which is non-unique). This is a simple exercise, which we skip.

Thus, we have a mapping
\beq
\mu: \ext{F}{E} \mapsto S(F,E)/(\rho R(F,E)).
\eeq
Let us prove that it is bijective.

Suppose that $\mu \alpha=0$. In the previous notations, this means that $\rho p \in \rho R(F,E)$, i.e. that $\rho p= \rho h$ for some $h\in R(F,E)$.  Let us set $p_1:=p-h$. Then $\rho p_1=0$, so $p_1$ is in fact linear: $p_1\in L(F,G)$. Besides, $\sigma p_1=\sigma p-\sigma h= I_F -0=I_F$. This means that $p_1$ is a continuous linear right inverse to $\sigma$, and thus $E$ is complemented in $G$ and hence $\alpha=0$. This shows injectivity of $\mu$.

Let now $\phi\in S(F,E)$. We equip the space of pairs $(x,y)\in E\times F$ with component-wise multiplication by scalars and addition defined as follows:
\beq
(x_1,y_1)+(x_2,y_2):=(x_1+x_2-\phi(y_1,y_2), y_1+y_2).
\eeq
Using the properties of factor systems described in Definition \ref{D:factors}, one can check that this defines a vector space structure on $G:=E\times F$. A norm on $G$ is introduced by defining its unit ball as the convex hall of the union of the two sets:
\beq
S_E:=\{(x,0)\,|\,\|x\|_E\leq 1\}
\eeq
and
\beq
S_F=\{(0,y)\,|\,\|y\|_F\leq 1\}
\eeq
This makes $G$ a normed, and in fact Banach space. Indeed, let $(x_n,y_n)$ be a Cauchy sequence in $G$. By construction, the natural mapping $(x,y)\mapsto y$ from $G$ to $F$ is linear, surjective, and of norm not exceeding $1$. Thus, the sequence $\{y_n\}$ is fundamental (Cauchy) in $F$ and there exists a limit $\lim\limits_{n\to\infty}y_n = y\in F$.

Let us now check the Cauchy property for $\{x_n\}$ in $E$. We have the following equalities:
\beq\label{E:xn0}
(x_n,y_n)-(0,y_n)=(x_n+0-\phi(y_n,-y_n),0)=(x_n,0))
\eeq
and
\beq\label{E:0yn}
(0,y_n)-(0,y_m)=(-\phi(y_n,-y_m)),y_n-y_m).
\eeq
(We used here that $\phi(y,-y)=0$, as it follows from conditions (1) and (2) of Definition \ref{D:factors}.)

Since ${y_n}$ is a Cauchy sequence and $\phi$ is continuous, we get
\beq
\lim\limits_{n,m\to\infty}\phi(y_n,-y_m)=\phi(y,-y)=0.
\eeq
Thus, due to (\ref{E:0yn}), $\{(0,y_n)\}$ is a Cauchy sequence in $G$. Indeed,
\beq
(-\phi(y_n,-y_m),y_n-y_m)=(-\phi(y_n,-y_m),0)+(0,y_n-y_m).
\eeq
Here both terms tend to zero in $G$, since
\beq\label{E:normineq}
\|(x,0)\|_G\leq \|x\|_E \mbox{ and } \|(0,y)\|_G\leq \|y\|_F.
\eeq
The relation (\ref{E:xn0}) implies that $(x_n,0)$ is a fundamental (Cauchy) sequence in $G$. We want to conclude that then $\{x_n\}$ is a Cauchy sequence in the norm of $E$. In fact, the norm induced from $G$ is equivalent to the internal norm of $E$. Indeed, one inequality has already been established above: $\|(x,0)\|_G\leq \|x\|_E$. To establish an inequality in the opposite direction, it is thus sufficient to show that the intersection of the unit ball in $G$ with the subspace $E$ is bounded. So, let $(x,0)\in G$ belongs to the unit ball $B$ of $G$. This means that
\beq
(x,0)=(\tilde{x},0)+ \sum\limits_{i=1}^n (0, y_i),
\eeq
such that
\beq
\sum\|y_i\|_F+\|\tilde{x}\|_E \leq 1.
\eeq
Using the definition of the addition in $G$, one gets
\beq
(x,0)=(\tilde{x}+\sum_{k=1}^{n-1}\phi(\sum_{i=1}^{k}y_i,y_{k+1}),\sum\limits_{i=1}^n (0, y_i)).
\eeq
We thus get
\beq
\|x\|_E\leq \|\tilde{x}\|_E+C(\phi)\sum\|y_i\|_F\leq 1+C(\phi),
\eeq
and hence $x$ belongs to the ball in $E$ of radius $1+C(\phi)$. This proves equivalence of the norms.

We thus have established that the embedding of $E$ into $G$ is an isomorphism and Cauchy sequence $\{(x_n,y_n)\}$ in $G$ produces Cauchy sequences $\{x_n\}$ in $E$ and $\{y_n\}$ in $F$. If now the limits of the latter sequences are $x\in E$ and $y\in F$ correspondingly, then $(x,y)$ is the limit of $\{(x_n,y_n)\}$ in $G$. Indeed, we have
\beq
(x,y)-(x_n,y_n)= (x-x_n-\phi(-y_n,-y), y-y_n).
\eeq
Then
\beq
\lim (y-y_n)=0, \lim (x-x_n)=0, \mbox{ and } \lim\phi(-y_n,y)=\phi(-y,y)=0.
\eeq
Adding to this the inequality
\beq
\|(x,y)\|_G\leq \max (\|x\|_e,\|y\|_F),
\eeq
we conclude that $\lim (x_n,y_n)=(x,y)$, and thus $G$ is a Banach space.

Let us collect what we have done so far: starting with $\phi$, we constructed an extension $\alpha\in\ext{F}{E}$
\beq
\short{E}{}{G}{}{F}.
\eeq
It remains to show that the class in $S(F,E)/(\rho R(F,E))$ that corresponds to $\alpha$ is indeed $\phi$.
It is clear from the construction of the space $G$ that one can choose a right inverse to the surjection $G\mapsto F$ as follows: $p: y \to (0,y)$. Then $\rho p(y_1,y_2)$ is the first component in $E\times F$ of the expression
\begin{align}
(0,y_1+y_2)-(0,y_1)-(0,y_2)=(0, y_1+y_2)-(-\phi(y_1,y_2),y_1+y_2))\\
= (\phi(y_1,y_2)+\phi(y_1+y_2,-y_1-y_2),0) = (\phi(y_1,y_2),0).
\end{align}
This proves that $\rho p =\phi$, which is the needed surjectivity and finishes the proof of the theorem\footnote{Modulo checking functoriality.}.
\end{proof}

\subsection{Vector space structure and functoriality of $Ext^1$}

The way it was defined, $Ext^1(E,F)$ was just a set. However, one can equip it with a vector space structure such that $Ext^1$ becomes a bi-functor from the category \textbf{Ban} of Banach spaces with bounded linear operators as morphisms\footnote{Higher $Ext^i$ functors can be also defined for any natural number $i$ (see, e.g., \cite{Palam,Cast19,3space_book}).}, into the category of vector spaces with linear mappings as morphisms. Namely,
Theorem \ref{T:Fact_ext} enables one to transfer linear structure from $S(F,E)$ onto $\ext{F}{E}$, thus turning the latter into a vector space. It also provides a left action of $L(E)$ and a right one of $L(F)$ on $\ext{F}{E}$ and allows for an easy check of functoriality\footnote{One can also act by operators from $L(E,X)$ and $L(X,F)$, where $X$ is another Banach space. The results reside in $\ext{F}{X}$ and $\ext{X}{E}$ respectively.}. We will provide later on independent of Theorem \ref{T:Fact_ext} definitions of these action and vector structure.

We collect basic properties of $Ext^1$ as follows:
\begin{theorem}\cite{Kuch76,3space_book}\label{T:action}
\begin{enumerate}
   \item The multiplications of elements of $\ext{F}{E}$ by operators from $L(E,X)$ and $L(X,F)$ is bi-linear.
  \item Multiplications by elements of $L(E)$ and $L(F)$ are correspondingly representation  and anti-representations of these algebras in $\ext{F}{E}$.
  \item $Ext^1$ is a bi-functor (contra-variant and co-variant with respect the first and the second argument correspondingly) from the category of Banach spaces with bounded linear operators as morphisms into the category of vector spaces with linear mappings as morphism.
  \item $Ext^1$ is the derived functor for $Hom$.
\end{enumerate}
\end{theorem}

\section{The three-representation problem}\label{S:main}
We are now well equipped to address the three-representation problem.

Let $T$ be a representation of a topological group $\mathcal{G}$ in a Banach space $G$, such that the mapping
\beq
(g,x)\in (\mathcal{G},G) \mapsto T(g)x\in G
\eeq
is continuous. Let also $E\subset G$ be a closed $T$-invariant subspace and $F:=G/F$ the corresponding quotient space. We denote by $T_1$ the restriction of the representation $T$ to $E$ and by $T_2$ the quotient-representation in $F$. We thus have a short exact sequence of Banach $\mathcal{G}$-moduli
\beq
\short{E}{ }{G}{ }{F}.
\eeq

\bed\indent
\begin{enumerate}
\item
We define actions of an element $g$ of $\mathcal{G}$ on $S(F,E)$ and $R(F,E)$ as follows:
\begin{eqnarray}
(g \phi)(x,y)& := & T_1(g)\phi(T_2(g)^{-1}x,T_2(g)^{-1}y),\\
(g h)(x) & := & T_1(g)h(T_2(g)^{-1}x).
\end{eqnarray}
\item We denote by $Z^1(\mathcal{G},R(F,E))$ the space of $1$-\textbf{co-cycles} on $\mathcal{G}$ with values in $R(F,E)$, i.e. mappings $M:\mathcal{G}\mapsto R(F,E)$ such that
    \beq
    M(g_1g_2)=g_1M(g_2)+M(g_1)
    \eeq
    and such that the mapping $(g,x)\in (\mathcal{G}\times F)\mapsto M(g)x\in E$ is continuous.

\item The space of $1$-\textbf{co-boundaries}, i.e. co-cycles of the form $M(g)=g h-h$ for some $h\in R(F,E)$ is denoted by $B^1(\mathcal{G},R(F,E)))$.
\item The spaces $Z^1(\mathcal{G},S(F,E))$ and $B^1(\mathcal{G},S(F,E)))$ are defined analogously.
\item The cohomology space $H^1(\mathcal{G},R(F,E)))$ is defined as follows:
\beq
H^1(\mathcal{G},R(F,E))):=Z^1(\mathcal{G},R(F,E)))/B^1(\mathcal{G},R(F,E))).
\eeq
The space $H^1(\mathcal{G},S(F,E)))$  is defined analogously.
\item A mapping $d:S(F,E)\mapsto Z^1(\mathcal{G},S(F,E)))$ is defined as
\beq
(d\phi)(g)=g\phi-\phi.
\eeq
\item The following quotient mapping is naturally defined:
\beq
i:Z^1(\mathcal{G},S(F,E)))\mapsto Z^1(\mathcal{G},S(F,E)))/B^1(\mathcal{G},R(F,E))).
\eeq
\end{enumerate}
\ed

Let us have now an extension of Banach $\mathcal{G}$-moduli
\beq
\short{E}{ }{G}{ }{F}.
\eeq
According to Theorem \ref{T:Fact_ext}, if one forgets the group action, it corresponds to an element $\Phi\in S(F,E)/\rho R(F,E)$. Let also $p:F\mapsto G$ be a Bartle-Graves mapping (see the explanation below the formulation of Theorem \ref{T:Fact_ext}).

Let us now consider the co-cycle
\beq
(\psi(g))(x):=T(g)p(T_2(g)^{-1}x)-p(x)
\eeq
and its class $\Psi \in H^1(\mathcal{G},R(F,E))$ (one can check, using the properties of the Bartle-Graves map, that $\Psi$ indeed ends up having values in $L(F,E)$).

We thus have a pair
\beq\label{E:pairs}
(\Phi,\Psi)\in \ext{F}{E}\times H^1(\mathcal{G},R(F,E))
\eeq
that corresponds to our extension of Banach $\mathcal{G}$-moduli.

\begin{theorem}\label{T:main}
The mapping introduced above is a bijection between the set of congruence classes of Banach $\mathcal{G}$-moduli extensions of $E$ by $F$
and the pairs (\ref{E:pairs}) such that
\beq\label{E:degree}
id\Phi=i\rho\Psi.
\eeq
\end{theorem}
Thus, cohomologies \textbf{with coefficients in spaces of nonlinear} operators arise. The ``degree of nonlinearity'' is controlled by the Banach space extension congruence class $\Phi$ through the condition (\ref{E:degree}). In particular, if $E$ has a direct complement in $G$, then $\Phi=0$ and thus a class from $H^1(\mathcal{G}, L(F,E))$ arises, which is well known (e.g., \cite{KirillovBook,Moore}).
\begin{proof}
As the above constructions show, both objects $\Phi$ and $\Psi$ arise from some Bartle-Graves maps from $F$ to $G$, while apriori these maps
might be different. So, let us assume that a Bartle-Graves map $p:F\to G$ gives rise to $\Psi$, while $q:F \to G$ leads to $\Phi$. One
computes then that
\begin{align}
\Psi(g,x)=T(g)p(T_2(g)^{-1}x) - p(x),\\
\Phi(x,y)=q(x+y)-q(x)-q(y).
\end{align}
The direct calculation shows now that
\beq\label{E:B1}
\begin{split}
d\Phi (g)(x,y)=T(g)p(T_2(g)^{-1}x+T_2(g)^{-1}y)+p(x)+p(y)\\
-(p(x+y)+T(g)p(T_2(g)^{-1}x)+T(g)p(T_2(g)^{-1}y)).
\end{split}
\eeq
One calculates that the same expression arises for $\rho\Psi$, with replacement of $p$ with $q$. Then subtraction of the two expressions shows that
$d\Phi-\rho\Phi$ has the same form (\ref{E:B1}), with $(p-q)$ arising where $p$ used to be. Since by the definition of the Bartle-Graves maps
their difference $(p-q)$ maps $F$ into $E$, we conclude that $$T(g)\circ(p-q)=T_1(g)\circ(p-q).$$ This implies a formula analogous to (\ref{E:B1})
with a mapping $(p-q)\in R(F,E)$ and replacement of the representation $T$ by its sub-representation $T_1$. Then it is just the matter of
direct observation of the result to see that that $i(d\Phi-\rho\Psi)$ belongs to $B^1(\mathcal{G},R(F,E))$.

The injectivity and surjectivity of the constructed correspondence are straightforward.
\end{proof}

\section{More on $EXT^1$ in the category of Banach spaces}\label{S:moreext}

As we have mentioned before, the factor system representation (\ref{E:factorsystem}) provides a vector space structure on $\ext{F}{E}$ and enables one to introduce actions of operators from $L(E)$ and $L(F)$ there. It also proves the bi-functor property of $Ext^1$. However, both of these can be done without  (\ref{E:factorsystem}), which we illustrate below. Both approaches lead to the same structures.

\subsection{Actions on $Ext^1$}

We start with introducing action on $Ext^1$ of operator between appropriate spaces (using the socalled pushout and pullback constructions).  Namely, let $X$ be a Banach space, $\alpha\in\ext{F}{E})$, $T\in L(E,X)$, and $S\in L(X,F)$. We will define now  elements $T\alpha\in\ext{F}{X}$ and $\alpha S\in \ext{X}{E}$.

Let $\alpha$ have as a representative the extension
\beq
\short{E}{i}{G}{\sigma}{F}.
\eeq
Consider the graph of $T$ in $G\oplus X$:
\beq
\Gamma:=\{(x,y)\in G\oplus X \, |\, x\in i(E), Ti^{-1}x=y\}.
\eeq
We define
\beq
G_1:=(G\oplus X)/\G.
\eeq
The natural embedding $X\mapsto G\oplus X$ induces an imbedding $i_1:X\mapsto G_1$. Analogously, the superposition of the natural projection $G\oplus X \mapsto G$ with $\sigma$ induces a surjection $\sigma_1: G_1\mapsto F$. Now $T\alpha\in\ext{F}{X}$ is defined as the equivalence class of the extension
\beq
\short{X}{i_1}{G_1}{\sigma_1}{F}.
\eeq

Let now $S\in L(X,F)$. We define the Banach space $G^1$ as the closed subspace in $G\oplus X$ as follows:
\beq
G^1:=\{(x,y)\in G\oplus X\, | \, \sigma x=S y\}.
\eeq
Then $G^1$ contains $i(E)\oplus {0}$, which defines the injection $i^1:E\mapsto G^1$. The natural projection of $G\oplus X$ onto $X$ produces a surjection $\sigma^1:G^1\mapsto X$. It is easy to check that the following sequence is exact:
\beq
\short{E}{i^1}{G^1}{\sigma^1}{X}.
\eeq
Its congruence class defines the element $\alpha S\in \ext{X}{E}$.

\subsection{More on the vector space structure and functoriality of $Ext^1$}

Let us now equip $\ext{F}{E}$ with a vector space structure.
Having two extensions $\alpha_1, \alpha_1\in\ext{F}{E}$, one can define the direct sum:
\beq
\short{E\oplus E}{\i_1\oplus \i_2}{G_1\oplus G_2}{\sigma_1\oplus \sigma_2}{F\oplus F}.
\eeq
This is not what we need, though, since we want to end up in $\ext{F}{E}$.
For doing so\footnote{A Bauer construction.}, we introduce two auxiliary operators $\nabla:E\oplus E\mapsto E$ and $\Delta: F\mapsto F\oplus F$ as follows:
\beq
\nabla(x,y):=x+y, \Delta(x):=(x,x).
\eeq

\bed\label{D:sum}
The sum of two extensions $\alpha_1,\alpha_2 \in \ext{F}{E}$ is
\beq
\nabla(\alpha_1\oplus \alpha_2)\Delta.
\eeq
\ed
\brem
It is easy to check that under this addition, $\ext{F}{E}$ becomes an abelian semi-group with the zero element provided by the class of ``trivial'' (i.e. split) extension
\beq
\short{E}{i}{E\oplus F}{\sigma}{F},
\eeq
where $i$ and $\sigma$ are the natural embedding into the first and $\sigma$ natural projection onto the second component correspondingly.
\erem
Multiplication of an element $\alpha\in\ext{F}{E}$ by a scalar $\lambda$ can be defined as $(\lambda I_E)\alpha$.

The result is the following well known proposition ( see, e.g. \cite{Kuch76,3space_book}):
\begin{theorem}
The above definitions turn $\ext{F}{E}$ into a vector space.
\end{theorem}

\subsection{$Ext$ and injective/projective resolvents}

The category \textbf{Ban} of Banach spaces is well known to have plenty of injective and projective objects  (e.g. $l_\infty$- and $l_1$- spaces), which allows for descriptions\footnote{See also Theorem 2 in \cite{Kuch76}.} of functors $Ext^i$ through the corresponding resolvents \cite{3space_book}. Indeed, any Banach space $E$ can be isomorphically embedded into an $\l_\infty(\G)$ space (choosing, for instance, as the set $\G$ the unit ball in $E^*$)
\beq
\short{E}{i_1}{l_\infty(\G_1)}{\sigma_1}{M}.
\eeq
 Any Banach space $F$ can be represented as a quotient space of $l_1(\G)$ (choosing, for instance, as the set $\G$ the unit ball in $F$)
\beq
\short{N}{i_2}{l_1(\G_2)}{\sigma_2}{F}.
\eeq
These representations lead to the following well known result, which we will use later in this text:
\begin{theorem}\label{T:injproj}(see \cite[Theorem 2]{Kuch76} for (\ref{E:proj}), as well as \cite{3space_book} and references therein.)
The following functor isomorphisms hold:
\begin{align}
\ext{F}{E}\approx L(N,E)/(L(l_1(\G_2)\circ i_2) \mbox{ and }\label{E:proj}\\
\ext{F}{E}\approx L(F,M)/(\sigma_1\circ L(F,l_\infty(\G_1))).
\end{align}
\end{theorem}

\section{Can $\ext{F}{E}$ be of a positive finite dimension?}\label{S:small}

We do not know\footnote{Albeit experts on Banach space theory might know the answer to this question.} whether there exist (infinite dimensional) Banach spaces $E$ and $F$ such that $0<\dim(\ext{F}{E})<\infty$. We describe below some results (also stated without a proof in \cite{Kuch76}) that provide several classes of spaces where such possibility is excluded.

\bed\label{D:bap}

A Banach space $E$ has \textbf{approximation property} if for any compact set $K\subset E$ there is a sequence of (bounded)
linear operators of finite rank on $E$ that converges to the identity uniformly on $K$.

A Banach space $E$ has the \textbf{bounded approximation property (BAP)} if there exists a net $K_\beta$ of uniformly bounded in norm finite-rank operators in $E$ that uniformly on compacts approximates the identity operator $I_E$.
\ed
One can find discussions of these properties, for instance, in \cite{FJP,Grot}.

\begin{theorem}\cite{Grot}\label{T:Grot}
A reflexive space that has the \textbf{AP} also has the \textbf{BAP}.
\end{theorem}
\begin{theorem}\cite[Theorem 3]{Kuch76}\label{T:bap}
Let one of the following conditions be satisfied:
\begin{enumerate}
  \item $E$ is reflexive\footnote{It seems that it would be sufficient to assume that $E$ is a complemented subspace in its bi-dual.} and has AP;
  \item $F$ has BAP and is isomorphic to a dual space\footnote{It seems that the statement still holds when $F$ is a complemented in its bi-dual and has BAP}.
\end{enumerate}
Then either $\ext{F}{E}=0$, or $\dim \ext{F}{E}=\infty$.
\end{theorem}
\begin{proof}
Let us start with proving the first statement. According to Theorem \ref{T:injproj}, if $\ext{F}{E}\neq \{0\}$, then the (continuous) restriction operator
\beq\label{E:restrict}
L(l_1(\G_2),E)\to L(N,E)
\eeq
is not surjective. If we show that the range of this operator is not closed, the following simple folklore consequence of the open mapping theorem, together with Theorem \ref{T:injproj} will imply that the space $\ext{F}{E}$ is infinite dimensional.
\begin{lemma}(Implicit in \cite{Krein})
Let $A: B_1\to B_2$ be a bounded linear operator between Banach spaces. If its range $R(A)$ is not closed, then
\beq\label{E:inf}
\dim (B_2/R(A))=\infty.
\eeq
(It is important to notice that \textbf{no closure of $R(A)$ is being taken in (\ref{E:inf}).})
\end{lemma}
\begin{proof}
First of all, one can assume that $R(A)$ is dense in $B_2$. Factoring out the kernel of $A$, we can also assume that $A$ is injective. Suppose that the dimension in question is finite and is equal to $n$. Consider and $n$-dimensional linear subspace $V\subset B_2$ such that the linear span of $V$ and $R(A)$ is the whole space $B_2$. Consider now the operator $A'$ from $B_1\oplus V$ to $R(A) + V=B_2$, defined as follows:
\beq
A':=\left(
      \begin{array}{cc}
        A & 0 \\
        0 & I_V \\
      \end{array}
    \right).
\eeq
This is an injective, continuous, and surjective operator from $B_1\oplus V$ onto $B_2$. Since then its inverse is also continuous, convergent sequences in $R(A)$ correspond to convergence of their pre-images in $B_1$. This contradicts the assumed non-closedness of $R(A)$.
\end{proof}
We can now return to the proof of the 1st instance of the theorem. The task is to show that the range of the operator (\ref{E:restrict}) is not closed. Suppose that it is closed. Then, according to the open mapping theorem, there exists a constant $C$ such that for any $T\in L(N,E)$ that is extendable to the whole space $\l_1(\G)$, there exists an extension $\tilde{T}$ with an estimate \beq
\|\tilde{T}\|\leq C\|T\|.
\eeq
Using the BAP property, which follows from the assumption and Theorem \ref{T:Grot}, one can show that then any bounded operator is extendable, which would contradict to the assumptions. Indeed, according to the Hahn-Banach theorem, all finite rank operators are extendable. Let now $T\in L(N,E)$ and $T_\alpha:=K_\alpha T$, where $K_\alpha$ are the operators from the BAP property. We have then $\|T_\alpha\|\leq C\|T\|$. Let $\tilde{T_\alpha}$ be such extensions that $\|\tilde{T_\alpha}\|\leq C\|T\|$. Let $\mathcal{P}$ be the (compact) Tikhonov product
\beq
\mathcal{P}:=\mathop{\sqcap}\limits_{x\in B}B_{E,C},
\eeq
where $B$ is the unit ball in $l_1(\G)$ and $B_{E,C}$ is the ball of radius $C$ in $E$ equipped with weak topology.  Consider the net $\{x_\alpha\}\subset \mathcal{P}$, where
\beq
x_\alpha:=\mathop{\sqcap}\limits_{x\in B}\tilde{T_\alpha} x,
\eeq
and $x\in B$. Let $\mathcal{X}$ be a limit point for $\{x_\alpha\}$. One defines now an operator $\tilde{T}$ as follows:
\beq
\mathcal{X}=\mathop{\sqcap}\limits_{x\in B}\tilde{T}x.
\eeq
One can check now that $\tilde{T}$ is the required extension of $T$, which finishes the proof of the first claim of the theorem.

In the second, case considerations are analogous, using the second statement of Theorem \ref{T:injproj}. Let $E=X^*$ and the following short exact sequence holds:
\beq
\short{N}{i}{\l_1(\G)}{\pi}{X}.
\eeq
Then we get the exact sequence
\beq
\short{E}{\pi^*}{\l_\infty(\G)}{i^*}{N^*}.
\eeq
According to Theorem \ref{T:injproj}, we have
\beq
\ext{F}{E}\approx L(F,N^*)/(i^*L(F,l_\infty(\G))).
\eeq
Like in the previous case, it is sufficient to prove that if $\ext{F}{E}\neq 0$, the range of the operator $i^*$ acting on $L(F,l_\infty(\G))$ is not closed. Suppose that this is not true and $R(i^*)$ is closed. Then, as before, there exists a constant $C>0$ such that for any $T\in i^*L(F,N^*)$ there exists $\tilde{T}\in L(F,l_\infty(\G))$ such that $i^*\tilde{T}=T$ and $\|\tilde{T}\|\leq C\|T\|$. Let now $T\in L(F,N^*$ be arbitrary and $K_\beta\in L(F)$ - operators from the BAP property. Consider the operators $T_\beta:=TK_\beta\in L(F,N^*)$ and their ``liftings'' $\tilde{T_\beta}\in L(F,l_\infty(\G))$. Each operator $\tilde{T_\beta}:F\mapsto l_\infty(\G)$ is represented by a bounded (uniformly with respect to $\beta$ and $\gamma$) family of functionals $f_{\beta,\gamma}\in F^*$. Due to weak-* compactness of the ball in $F^*$, the previous construction enables one to find an operator $\mathcal{T}$ that is a limit point  in weak-* topology. One can check its boundedness and the equality $i^*\mathcal{T}=T$. This contradicts the assumption that $\ext{F}{E}$, and thus $L(F,N^*)/(i^*L(F,l_\infty(\G)))$ is not zero.
\end{proof}

A different class of cases when one can prove that $\ext{F}{E}$ is either zero or infinite-dimensional can be provided as follows.

According to part (2) of Theorem \ref{T:action}, the algebras $L(E)$ and $L(F)$ act on $\ext{F}{E}$. This action is unital, i.e. $I_E$ (resp. $I_F$) acts as identity.

\begin{lemma}
If $E\oplus E$ is isomorphic to $E$, then there are no unital (i.e., with $I_E$ acting as an identity) action of $L(E)$ on any finite-dimensional vector space $V$.
\end{lemma}
\begin{proof}
We denote for an operator $T\in L(E)$ its action on $V$ as $\widetilde{T}$. Let $E_1\oplus E_2\approx E$, $E_1, E_2 \approx E$, and $P_j$ be the corresponding projector onto $E_j$ in $E$. Consider an isomorphism $\tau: E\mapsto E_1$ and its action $\tilde{\tau}$ on the space $V$ of presentation. Then $\tau_1:=\tau^{-1}P_1$ is a left inverse to $\tau$. Hence, $\widetilde{\tau_1}:=\widetilde{\tau^{-1}P_1}$ is a left inverse to $\tilde{\tau}$. Due to finite dimensionality, it is also a right inverse. This implies that
\beq
\widetilde{P_2}=\widetilde{I_E -P_1}=\widetilde{I_E-\tau\tau_1}=0.
\eeq
The same consideration with $E_2$ instead of $E_1$ shows that $\widetilde{P_1}=0$. Thus,
\beq
I_V=\widetilde{I_E}=\widetilde{P_1+P_2}=0,
\eeq
which contradicts to the property of being unital.
\end{proof}.%

The lemma implies the following result:
\begin{theorem}\cite[Theorem 3]{Kuch76}\label{T:ideal}
Let  $E\oplus E$ be isomorphic to $E$ or $F\oplus F$ be isomorphic to $F$.
Then either $\ext{F}{E}=0$, or $\dim \ext{F}{E}=\infty$.
\end{theorem}

\section{Remarks and conclusions}\label{S:rem}

\begin{enumerate}
\item In the complemented case the problem of group representation extensions has been resolved long time ago \cite{KirillovBook,Kuch76,Moore}. Our main aim was to treat the general case. One can also consult with \cite{CastGr} concerning groups actions on twisted sums of Banach spaces.
  \item \textbf{Extensions in the Hilbert case. Lessons from \cite{Enflo}.}
If in the extension
\beq
\short{E}{}{H}{}{F}
\eeq
the space $H$ is (isomorphic to) a Hilbert space, then one concludes that both $E$ and $F$ are Hilbert spaces and the extension is trivial and thus represents the zero element in $\ext{F}{E}$. This begs the following natural question:
\begin{center}
\textbf{If $H$ is (isomorphic to) a Hilbert space, is $\ext{H}{H}$ equal to zero?}
\end{center}
To put it differently, if a subspace $E$ of a Banach space $G$ is isomorphic to a Hilbert space, and the corresponding quotient space $G/E$ also is, is then $G$ necessarily also a Hilbert space?
This question was answered in the negative\footnote{Although the $Ext$ functor was not eplicitly mentioned there.} in the famous paper \cite{Enflo},
where it was shown that
$$\ext{l_p}{l_p}\neq 0$$
for any $1<p<\infty$ (in particular, for $p=2$). The interesting observation is that, although no homological techniques were declared in
\cite{Enflo}, the proof there is best understood in terms of the factor system representation of $Ext^1$, which is described in Section \ref{S:factor}.
One can also notice, as the brief proof below shows, that the statement is about the local, rather than global geometry of Banach spaces, so the following little bit (just superficially) more general result holds\footnote{See also Theorem \ref{T:Bconv} by J.~Castillo.}.

      \bed We say that a Banach space $E$ contains uniformly complemented copies of $l^n_p$ for any $n$, if for any natural $n$ there exists an $n$-dimensional subspace $E_n\subset E$ and a projector $P_n:E\mapsto E_n$ such that $\|P_n\|\leq C$ and $d(E_n,l^n_p)\leq C$, where $d$ is the Banach-Mazur distance and $l^n_p$ is the $n$-dimensional $l_p$-space.
      \ed
      It is known that this property being satisfied for some $1<p<\infty$ is also satisfied for $p=2$.

      \begin{theorem}\label{T:Sp}
      Let $E$ and $F$ contain uniformly complemented copies of $l^n_d$ for any natural $n$ and be infinite dimensional. Then
      \beq
      \ext{F}{E}\neq 0.
      \eeq
      \end{theorem}
      \begin{proof}
      If $\ext{F}{E}=0$, then, according to Theorem \ref{T:Fact_ext}, the operator $\rho:R(F,E)\mapsto S(F,E)$ is surjective. This means that it defines an isomorphism
      \beq
      R(F,E)/L(F,E) \approx S(F,E).
      \eeq
      By the open mapping theorem, there is a constant $C$ such that for any $h\in R(F,E)$ there exists a linear operator $H\in L(F,E)$ such that $\|h-H\|\leq \|\rho h\|$. According to the definition of the norm in $S(F,E)$, one can see that the inequality $\|\rho h\|\leq 1$ is equivalent to
      \beq\label{E:increase}
      \|\sum_{i=1}^nh(x_i)\|\leq \sum_{i=1}^n \|x_i\|
      \eeq
      for any collection of vectors $x_i$ such that $\sum x_i=0$.

      The crucial moment in \cite{Enflo} was finding a technique of increasing the distance of $h$ to the space $L(F,E)$ of linear operators without breaking the inequality (\ref{E:increase}), which immediately leads to the statement of the theorem. Here is the construction from \cite{Enflo}. Let $h$ belongs to $R(l^m_2,l^m_2)$. Consider the operator $\Delta h\in R(l^{2m}_2,l^{2m}_2)$ acting as follows:
      \beq
      \Delta h(x,y):=\left(h(x),h(y),\frac{x\|y\|}{\sqrt{\|x\|^2+\|y\|^2}}\right),
      \eeq
      where $x,y\in l^m_2$. A simple calculation shows that the inequality (\ref{E:increase}) holds for $\Delta h$, if it does for $h$, while iteration of this construction shows that
      \beq
      dist (\Delta^k h, L(l_2,l_2)) \mathop{\to}\limits_{k\to\infty} \infty.
      \eeq

      Let now $E$ and $F$ satisfy the conditions of the theorem. Thus, there for any $n$ there exist $n$-dimensional subspaces $E_n\subset E$ and $F_n\subset F$ and projectors $P_n:E\to E_n, Q_n:F\to F_n$, such that
      $$
      \max{d(E_n,l_2^n), d(F_n,l_2^n), \|P_n\|, \|Q_n\|} < C.
      $$
      We thus can ``rethink'' $\Delta^k h$ as elements of $R(F_n,E_n)$ with an appropriate value of $n$. Considering now $H_k:=\Delta^k h P_n$, we conclude that $H_k\in R(F,E)$, $\|\sum_{i=1}^rH(k(x_i)\| \leq \sum \|x_i\|$, and
      $$
      \mathop{dist} (H_k, L(F,E)) \to \infty,
      $$
      which finishes the proof of the theorem.
      \end{proof}

The author has been informed by Prof. J.~M.~F.~Castillo that the locality of this result can be made much more explicit and general, see e.g. \cite[Theorems 3 and 4]{CabelloHouston} and \cite{KaltonPelc}. He also showed that a better and more general formulation instead of Theorem \ref{T:Sp} can be proven the same way:
\begin{theorem}\label{T:Bconv}
Given two $B$-convex spaces\footnote{See \cite{Beck,Talagr}} $X$ and $Y$, the space $Ext^1(X,Y)$ is nonzero.
\end{theorem}

An immediate corollary of Theorems \ref{T:injproj}, \ref{T:bap}, and \ref{T:Sp} is:

\begin{theorem}
Let $E$ and $F$ be infinite-dimensional $\mathcal{L}_2$-spaces satisfying one of the conditions of Theorems \ref{T:injproj} or \ref{T:bap} (e.g., $E\approx F \approx l_2$). Then
\beq
\dim \ext{F}{E}=\infty.
\eeq
\end{theorem}

\item As Professor J.~Castillo has mentioned to the author, the statement of Theorem \ref{T:ideal} follows immediately from the formula
\beq
\ext{F\oplus F}{E}=\ext{F}{E}\oplus\ext{F}{E}
\eeq
and its analog for $E\oplus E$.

This is a nice derivation. However our proof is probably worth more than the result, since it is based on absence of unital finite dimensional representations of $L(E)$, which \textbf{might} turn out to be useful in other situations.

  \item \textbf{$\ext{F}{E}$ can be equal to zero} for some infinite-dimensional Banach spaces $E$ and $F$. Quite a few examples of various
  kinds are known. Here are just some of them:
      \begin{enumerate}
        \item If $E$ is injective or $F$ is projective in a category of Banach spaces closed under extensions, then
            \beq
            \ext{F}{E}=0.
            \eeq
        For instance, Sobczyk's Theorem (e.g., \cite{Sob,Veech,Goldb}) implies that
            \beq
            \ext{F}{c_0}=0
            \eeq
            for any separable space $F$.
        \item J. Bourgain has provided \cite{Bourg} an example of non complemented subspace $B$ in $l_1$, isomorphic to $l_1$. Let $E:l_1/B$. Then one can show that
            \beq
            \ext{l_2}{E}=0.
            \eeq
        \item J. Lindenstrauss proved (without using the notion of $Ext$) that if $F$ is an $\mathcal{L}_1$-space and $E$ is complemented in its bidual, then
            \beq
            \ext{F}{E}=0.
            \eeq
      \end{enumerate}

  \item The author wants to clarify again that the results of this text were announced
  in 1976 in the $1.5$ pages long brief communication paper \cite{Kuch76}.
  To the best of author's knowledge, the proof of the solution of Kirillov's three-representation problem has never been published.

  The paper \cite{Kuch76} also contained the factor system representation of
  functor $Ext^1$ in \textbf{Ban}, which preceded by several years the well known by now somewhat similar,
  but still
  different twisted-sum construction of \cite{KaltonPeck} (see also \cite{3space_book} and
  references therein). It also contained possibly the first representation of $Ext^1$ in
  \textbf{Ban} through projective resolvents, as well as results excluding in various cases
  possibility of existence of $\ext{F}{E}$ of a non-zero finite dimension. The author has decided to
  provide details of considerations of \cite{Kuch76} here. It is clear that after the 45 years of
  enormous developments in the Banach space theory since \cite{Kuch76} appeared, some of the results presented could be (and most probably have been) extended or rediscovered. The author, being not an expert on the current state of
  the field, cannot say for sure.
\end{enumerate}

\section{Acknowledgments}
The author thanks J.~F.~M.~Castillo, W.~B.~Johnson, and A.~A.~Pankov for extremely useful comments, corrections, and references. Thanks
also go to the NSF for the support through the DMS grants \# 1517938 and \# 2007408.

\end{document}